\documentclass[11pt, a4paper]{article}
\usepackage{times}
\usepackage{a4wide}
\usepackage{hyperref}
\usepackage[british]{babel}
\usepackage{enumerate, longtable}
\usepackage{amsmath, amscd, amsfonts, amsthm, amssymb, latexsym, comment, stmaryrd, graphicx}
\usepackage[T1]{fontenc}
\usepackage[latin1]{inputenc}

\newtheorem{thm}{Theorem}[section]
\newtheorem{cor}[thm]{Corollary}
\newtheorem{lem}[thm]{Lemma}
\newtheorem{prop}[thm]{Proposition}
\theoremstyle{definition}

\theoremstyle{remark}
\newtheorem{rem}[thm]{Remark}

\numberwithin{equation}{section}

\newcommand{\NN}{\mathbb{N}}
\newcommand{\PP}{\mathbb{P}}

\newcommand{\Real}{\mathrm{Re}}

\begin{document}
\title{Equidistribution of Signs for Modular Eigenforms of Half Integral Weight}
\author{Ilker Inam\footnote{Uludag University, Deparment of Mathematics, Faculty of Arts and Sciences, 16059 Gorukle, Bursa, Turkey, ilker.inam@gmail.com}
and Gabor Wiese\footnote{Universit\'e du Luxembourg,
Facult\'e des Sciences, de la Technologie et de la Communication,
6, rue Richard Coudenhove-Kalergi,
L-1359 Luxembourg, Luxembourg, gabor.wiese@uni.lu}}

\maketitle

\begin{abstract}
Let $f$ be a cusp form of weight $k+1/2$ and at most quadratic nebentype character whose Fourier coefficients $a(n)$ are all real. We study an equidistribution conjecture of Bruinier and Kohnen for the signs of $a(n)$. We prove this conjecture for certain subfamilies of coefficients that are accessible via the Shimura lift by using the Sato-Tate equidistribution theorem for integral weight modular forms. Firstly, an unconditional proof is given for the family $\{a(tp^2)\}_p$ where $t$ is a squarefree number and $p$ runs through the primes. In this case, the result is in terms of natural density. To prove it for the family $\{a(tn^2)\}_n$ where $t$ is a squarefree number and $n$ runs through all natural numbers, we assume the existence of a suitable error term for the convergence of the Sato-Tate distribution, which is weaker than one conjectured by Akiyama and Tanigawa. In this case, the results are in terms of Dedekind-Dirichlet density.\\
\textbf{Mathematics Subject Classification (2010).} Primary 11F37; Secondary 11F30.\\
\textbf{Keywords.} Forms of half-integer weight, Shimura lift, Fourier coefficients of
automorphic forms, Sato-Tate equidistribution
\end{abstract}

\section{Introduction}
The signs of the coefficients of half-integral weight modular forms have attracted
some recent attention. In particular, Bruinier and Kohnen conjectured an equidistribution of
signs. Using the Shimura lift and the Sato-Tate equidistribution theorem we obtain
results towards this conjecture.

Throughout this paper, the notation is as follows.
Let $k \ge 2$, $4 \mid N$ be integers, $\chi$ a Dirichlet character modulo~$N$ s.t.\ $\chi^2=1$ and
let $f=\sum_{n=1}^{\infty}a(n)q^n \in S_{k+1/2}(N, \chi)$ be a non-zero cuspidal
Hecke eigenform (for operators $T_{p^2}$ for primes $p \nmid N$)
of weight $k+\frac{1}{2}$ (see \cite{Shi}, p.~458) with real coefficients.
Let $F_t=\sum_{n=1}^{\infty}A_t(n)q^n \in S_{2k}(N/2, \chi^2)$ be the Hecke eigenform
(for operators $T_p$ for primes $p \nmid N$)
of weight~$2k$ corresponding to~$f$ under the Shimura lift for a fixed squarefree $t$
such that $a(t) \neq 0$ (see Section~\ref{sec:background}).
We work under the assumption that $F_t$ does not have complex multiplication.

The question about the sign changes of coefficients of half integral weight modular forms has been asked by Bruinier and Kohnen in~\cite{BK} and there it was shown that the sequence $\{a(tn^2)\}_{n\in \NN}$ of coefficients of half integral weight modular forms has infinitely many sign changes under the additional hypothesis that a certain $L-$function has no zeros in the interval $(0,1)$ (later this hypothesis has been removed in \cite{Kohnen}). In addition to this partly conditional result, the authors came up with a suggestion supported by some numerical experiments. They claimed that half of the coefficients are positive among all non-zero coefficients due to observations on computations made for one example of weight $11/2$ given in \cite{KZ} and for one example of weight $3/2$ given in \cite{Du} for indices up to $10^6$. This claim is also supported by $5$ different examples of weight 3/2 in computations done in the preparation process of \cite{I} for indices up to $10^7$. In \cite{KLW}, Kohnen, Lau and Wu also study the sign change problem on specific sets of integers. They establish lower bounds of the best possible order of magnitude for the number of those coefficients that have the same signs. These give an improvement on some results in \cite{BK} and \cite{Kohnen}.
Although the problem of the equidistribution of signs has been mentioned informally in \cite{BK}, the conjecture was only given in \cite{KLW} formally.

In this paper we improve the previously known results towards the Bruinier-Kohnen conjecture by proving the equidistribution of the signs of $\{a(tn^2)\}_n$ for $n$ running through the set of primes (see Theorem~\ref{thm:prime}) and through the set of all positive integers (see Corollary~\ref{cor:squarefree2}). The former result is formulated in terms of natural density, the latter for Dedekind-Dirichlet density. For the latter result, we have to assume a certain error term for the convergence of the Sato-Tate distribution for integral weight modular forms, which is weaker than an error term conjectured by Akiyama and Tanigawa~\cite{AT}.
Relying on the Shimura lift as in the papers cited above, the techniques of the present article also do not extend to study the equidistribution of the signs of $\{a(n)\}_n$, when $n$ runs through the squarefree integers.

The ideas and techniques of the present paper have been adapted by Narasimha Kumar to $q$-exponents of generalised modular functions~\cite{Ku}.
In a sequel to this work~\cite{AIG}, written together with Sara Arias-de-Reyna, we extend the main theorem to the case when $F_t$ has complex multiplication and we also weaken the assumption on the error bound.

\section{Density and regular sets of primes}\label{sec:density}

Let $\PP$ denote the set of all prime numbers. We first recall some definitions.
Let $S \subseteq \PP$ be a set of primes. It is said to have {\em natural density}~$d(S)$
if the limit $\lim_{x \to \infty} \frac{\pi_S(x)}{\pi(x)}$ exists and is equal to~$d(S)$,
where $\pi(x) = \#\{ p \le x  \;|\; p \in \PP \}$ is the prime number counting function
and $\pi_S(x) := \#\{p \le x \;|\; p \in S\}$.
The set $S$ is said to have {\em Dirichlet density $\delta(S)$}
if the limit $\lim_{z \to 1^+} \frac{\sum_{p\in S}\frac{1}{p^z}}{\log\big(\frac{1}{z-1}\big)}$
exists and is equal to~$\delta(S)$ (the limit $\lim_{z \to 1^+}$ is defined via sequences of real
numbers tending to~$1$ from the right). If $S$ has a natural density, then it also has
a Dirichlet density and the two coincide. As in \cite{Na}, p.~343f, we call $S$ {\em regular} if
there is a function $g(z)$ holomorphic on $\Real(z) \ge 1$ such that
$\sum_{p\in S}\frac{1}{p^z} = \delta(S) \log\big(\frac{1}{z-1}\big) + g(z)$.
Note that $\PP$ is a regular set of primes of Dirichlet density~$1$.
We collect some more straight forward properties of regular sets of primes in the following lemma.

\begin{lem}\label{lem:regular}
\begin{enumerate}[(a)]
\item Let $S$ be any set of primes such that
the series $\sum_{p\in S} \frac{1}{p}$ converges to a finite value. Then $S$
has a Dirichlet density equal to~$0$.
\item Let $S$ be a regular set of primes. Then the Dirichlet density of $S$ is~$0$ if
and only if the series $\sum_{p\in S} \frac{1}{p}$ converges to a finite value.
\item Let $S_1,S_2$ be two regular sets of primes having the same Dirichlet density
$\delta(S_1)=\delta(S_2)$. Then the function
$\sum_{p \in S_1} \frac{1}{p^z} - \sum_{q \in S_2} \frac{1}{q^z}$
is analytic on $\Real(z) \ge 1$.
\end{enumerate}
\end{lem}

For our application in Section~\ref{sec:all} we need regular sets of primes.
Hence, we now include a proposition showing that if a set of primes $S$ has a natural density
and the convergence satisfies a certain error term, then it is regular.
This is may be known, but, we are not aware of any reference; so we include a full proof.

\begin{prop}\label{prop:regular}
Let $S \subseteq \PP$ have natural density $d(S)$. Call $E(x) := \frac{\pi_S(x)}{\pi(x)} - d(S)$
the error function. Suppose that there are $\alpha > 0$,
$C > 0$ and $M > 0$ such that for all $x > M$ we have $|E(x)| \le C x^{-\alpha}$.

Then $S$ is a regular set of primes.
\end{prop}

\begin{proof}
We will use the notation
$D_S(z) := \sum_{p \in S} \frac{1}{p^z}$ and $D(z) := \sum_{p \in \mathbb{P}} \frac{1}{p^z}$.
We abbreviate $d := d(S)$, put $g(x) := E(x)\pi(x)$
and $f(z) := \sum_{n=2}^\infty g(n) \big(\frac{1}{n^z} - \frac{1}{(n+1)^z}\big)$.
As $g(x)$ is a step function with jumps only at integers, we have
$$ f(z) = z \cdot \sum_{n=2}^\infty g(n) \int_{n}^{n+1} \frac{1}{x^{z+1}} dx
 = z \cdot \int_{2}^{\infty} \frac{g(x)}{x^{z+1}} dx.$$
From Theorem~29 of~\cite{Rosser}, we know that $\pi(x) < \frac{x}{\log(x)-4}$ for $x > 55$,
yielding $|g(x)| \le C\cdot\pi(x)\cdot x^{-\alpha} \le C\cdot\frac{x^{1-\alpha}}{\log(x)-4}$.
Thus for $\Real(z)>1-\frac{\alpha}{2}$ we get
$$|\int_{55}^{\infty}\frac{g(x)dx}{x^{z+1}}| \le \int_{55}^{\infty}\frac{|g(x)|}{x^{\Real(z)+1}}dx
\le C \int_{55}^{\infty}\frac{1}{x^{1+\frac{\alpha}{2}}}dx.$$
As the last integral is convergent, we conclude that $f(z)$ is an analytic function on
$\Real(z) \ge 1$. Since for $\Real(z)>1$ we have
$$D_S(z) = \sum_{n=2}^\infty \underbrace{\big(d \pi(n) - g(n) \big)}_{=\pi_S(n)}
\big(\frac{1}{n^z} - \frac{1}{(n+1)^z}\big)= d D(z) + f(z)$$
the proposition follows as~$\PP$ is regular of density~$1$.
\end{proof}

Finally we recall a notion of density on sets of natural numbers analogous to that of Dirichlet density.
A subset $A \subseteq \NN$ is said to have {\em Dedekind-Dirichlet density $\delta(A)$} if the limit
$\lim_{z \to 1^+} (z-1)\sum_{n=1, n \in A}^{\infty} \frac{1}{n^z}$ exists and is equal to $\delta(A)$.
If $A \subseteq \NN$ has a natural density (defined in the evident way), then
it also has a Dedekind-Dirichlet density, and they are the same.

\section{Shimura lift and Sato-Tate equidistribution}\label{sec:background}

We continue to use the notation described in the introduction.
We now summarise some properties of the Shimura lift that are proved in \cite{Shi}
and Chapter 2 of \cite{Niwa}.
There is the direct relation between the Fourier coefficients of~$f$ and those
of its lift~$F_t$, namely
\begin{equation}\label{eq:1}
A_t(n):=\sum_{d|n}^{} \chi_{t,N}(d)d^{k-1}a(\frac{tn^2}{d^2}),
\end{equation}
where $\chi_{t,N}$ denotes the character
$\chi_{t,N}(d):=\chi(d) \left(\frac{(-1)^{k}N^2t}{d} \right)$.
As we assume $f$ to be a Hecke eigenform for the Hecke operator $T_{p^2}$,
$F_t$ is an eigenform for the Hecke operator $T_p$, for all primes $p\nmid N$. In fact, in this case
$F_t=a(t)F$, where $F$ is a normalised Hecke eigenform independent of~$t$.
Moreover, from the Euler product formula for the Fourier coefficients of half integral weight modular forms, one obtains the multiplicativity relation for $(m,n)=1$
\begin{align}\label{coef}
a(tm^2)a(tn^2)=a(t)a(tm^{2}n^{2}).
\end{align}
Note that the assumption that $\chi$ be (at most) quadratic implies that
$F_t$ has real coefficients if $f$ does.
Furthermore, the coefficients of $F_t$ satisfy the Ramanujan-Petersson bound
$|\frac{A_t(p)}{a(t)}| \leq 2p^{k-1/2}$. We normalise them by letting
$$B_t(p):=\frac{A_t(p)}{2a(t)p^{k-1/2}} \in [-1,1].$$ 
Recall now that we are assuming $F_t$ without complex multiplication.
One defines the {\em Sato-Tate measure} $\mu$ to be the probability measure on the interval
$[-1,1]$ given by $\frac{2}{\pi} \sqrt{1-t^2}dt$.

Theorem B of \cite{BLGHT}, case 3 with $\zeta=1$, gives the important
{\em Sato-Tate equidistribution theorem} for $\Gamma_0(N)$, which applies in particular
to $F= \frac{F_t}{a(t)}$.
\begin{thm}[Barnet-Lamb, Geraghty, Harris, Taylor] \label{thm:ST}
Let $k \ge 1$ and let $F=\sum_{n \geq 1}^{} A(n)q^n$ be a normalised cuspidal Hecke eigenform of weight~$2k$ for $\Gamma_0(N)$ without complex multiplication.
Then the numbers $B(p) = \frac{A(p)}{2p^{k-1/2}}$ are $\mu$-equidistributed in $[-1,1]$,
when $p$ runs through the primes not dividing $N$.
\end{thm}

This has the following simple corollary, which we formulate for the Shimura lift~$F_t$.
\begin{cor}\label{cor:ST}
Let $[a,b] \subseteq [-1,1]$ be a subinterval and
$S_{[a,b]}:=\{p \textnormal{ prime } |\; p \nmid N, B_t(p) \in [a,b] \}$.
Then $S_{[a,b]}$ has natural density equal to $\frac{2}{\pi}\int_{a}^{b}\sqrt{1-t^2}dt$.
\end{cor}

\section{Equidistribution of signs for $\{a(tp^2)\}_{p \textnormal{ prime }}$}\label{sec:prime}

Our main unconditional result is the following theorem.

\begin{thm}\label{thm:prime}
We use the assumptions from the introduction, in particular, $F_t$ has no complex multiplication.
Define the set of primes
$$\PP_{> 0}:=\{p \in \PP | a(tp^2)  > 0 \}$$
and similarly $\PP_{< 0}$, $\PP_{\ge 0}$, $\PP_{\le 0}$, and $\PP_{=0}$
(depending on $f$ and~$t$).

Then the sets $\PP_{>0}$, $\PP_{<0}$, $\PP_{\ge 0}$, $\PP_{\le 0}$ have natural density~$1/2$
and the set $\PP_{=0}$ has natural density~$0$.
\end{thm}

\begin{proof}
Denote by $\pi_{>0}(x) := \#\{p \le x \;|\; p \in \PP_{>0}\}$ and similarly
$\pi_{<0}(x)$, $\pi_{\ge 0}(x)$, $\pi_{\le 0}(x)$, and $\pi_{=0}(x)$.
Since dividing $f$ by $a(t)$ does not affect the assertions of the theorem,
we may and do assume $a(t)=1$. In that case $F_t$ is a normalised eigenform.
Equation~\eqref{eq:1} specialises to
$a(tp^{2})=A_t(p)-\chi_{t,N}(p)p^{k-1}$
for all primes~$p$, implying the equivalence:
$$a(tp^{2}) > 0 \Leftrightarrow B_t(p) > \frac{\chi_{t,N}(p)}{2\sqrt{p}}.$$
The idea is to use the Sato-Tate equidistribution to show that
$|A_t(p)|$ dominates the term $\chi_{t,N}(p)p^{k-1}$ for `most' primes.
Let $\epsilon >0$. Since for all $p > \frac{1}{4\epsilon^2}$ one has
$|\frac{\chi_{t,N}(p)}{2\sqrt{p}}| = \frac{1}{2 \sqrt{p}}  <\epsilon$,
we obtain
\begin{equation}\label{eq:up}
 \pi_{>0}(x) + \pi(\frac{1}{4\epsilon^2}) \ge
   \#\{p \le x \textnormal{ prime }| \; B_t(p) > \epsilon\}.
\end{equation}
By Corollary~\ref{cor:ST}, we have
$\lim_{x \to \infty} \frac{ \# \{p\le x \textnormal{ prime }| \; B_t(p) > \epsilon\}}{\pi (x)}
= \mu ([\epsilon, 1])$.
This implies that
$$\liminf_{x \to \infty} \frac{\pi_{>0}(x)}{\pi (x)} \ge \mu ([\epsilon, 1])$$
for all $\epsilon > 0$, whence we can conclude
$\liminf_{x \to \infty} \frac{\pi_{>0}(x)}{\pi (x)} \ge\mu ([0, 1]) = \frac{1}{2}$.

A similar argument yields
$\liminf_{x \to \infty} \frac{\pi_{\le 0}(x)}{\pi (x)} \ge\mu ([0, 1]) = \frac{1}{2}$.
Using $\pi_{\le 0}(x) = \pi(x) - \pi_{> 0}(x)$ gives
$\limsup_{x \to \infty} \frac{\pi_{>0}(x)}{\pi (x)} \le \mu ([0, 1]) = \frac{1}{2}$,
thus showing that $\lim_{x \to \infty} \frac{\pi_{>0}(x)}{\pi (x)}$ exists and is equal
to~$\frac{1}{2}$, whence by definition $\PP_{>0}$ has natural density~$\frac{1}{2}$.
The arguments for $\PP_{<0}(x)$, $\PP_{\ge 0}(x)$, $\PP_{\le 0}(x)$ are exactly the same,
and the conclusion for $\PP_{=0}$ immediately follows.
\end{proof}

We next show that the sets of primes in Theorem~\ref{thm:prime} are regular if the
Sato-Tate equidistribution converges with a certain error term.
A stronger error term was conjectured by Akiyama-Tanigawa
(see \cite{Mazur}, Conjecture 2.2, and Conjecture 1, on p.~1204 in \cite{AT}).

\begin{thm}\label{thm:prime-regular}
We make the same assumptions as in Theorem~\ref{thm:prime}.
We additionally assume that there are $C>0$ and $\alpha > 0$ such that for all
subintervals $[a,b] \subseteq [-1,1]$ one has
$$ \left| \frac{ \# \{p\le x \textnormal{ prime }| \; \frac{A_t(p)}{a(t)2p^{k-1/2}} \in [a,b]\}}{\pi (x)}
 - \mu([a,b]) \right| \le \frac{C}{x^\alpha}.$$

Then the sets $\PP_{>0}$, $\PP_{<0}$, $\PP_{\ge 0}$, $\PP_{\le 0}$, and $\PP_{=0}$ are regular sets of primes.
\end{thm}

\begin{proof}
We start as in the proof of Theorem~\ref{thm:prime} up to Equation~\eqref{eq:up}
and plug in the error term to get for all $\epsilon>0$ and all $x>0$
\begin{multline*}
\pi_{>0}(x) + \pi(\frac{1}{4\epsilon^2}) 
 \ge \#\{p \le x \textnormal{ prime }| \; B_t(p) > \epsilon\}\\
 \ge - C\pi(x)x^{-\alpha} + \mu([\epsilon,1]) \pi(x)
 \ge - C\pi(x)x^{-\alpha} + (\frac{1}{2} - \mu([0,\epsilon])) \pi(x).
\end{multline*}
Using the estimates $\mu([0,\epsilon])=\int_0^\epsilon \sqrt{1-t^2} dt \le \epsilon$ and
$\frac{x}{\log(x)+2} \le \pi(x) \le \frac{x}{\log(x)-4}$, which is valid for $x>55$
(see \cite{Rosser} Theorem 29 A on p.~211), we get
$$ \frac{\pi_{>0}(x)}{\pi(x)} - \frac{1}{2}
\ge -\big(C x^{-\alpha} + \epsilon +
  \frac{(\log(x)+2)}{-4 x \epsilon^2 (\log(4\epsilon^2) +4)} \big) \ge - C_1 x^{-\alpha}$$
for some $C_1>0$ and all $x$ big enough, where for the last inequality we set
$\epsilon = x^{-2\alpha}$.
The same argument with $\pi_{<0}(x)$ yields
$$ \frac{\pi_{<0}(x)}{\pi(x)} - \frac{1}{2} \ge - C_1 x^{-\alpha}.$$
Using $\pi_{\ge 0}(x) = \pi(x) - \pi_{<0}(x)$, we get
$\frac{\pi_{>0}(x)}{\pi(x)} - \frac{1}{2} \le \frac{\pi_{\ge 0}(x)}{\pi(x)} - \frac{1}{2}
\le C_1 x^{-\alpha}$.
Thus, one has
$$ \left| \frac{ \pi_{>0}(x)}{\pi (x)} - \frac{1}{2} \right| \le \frac{C_1}{x^\alpha}.$$
Proposition~\ref{prop:regular} now implies that $\PP_{>0}$ is a regular
set of primes. The regularity of $\PP_{<0}$ is obtained in the same way, implying
also the regularity of $\PP_{=0}$, $\PP_{\ge 0}$ and $\PP_{\le 0}$.
\end{proof}

\begin{rem}
Assume the setup of Theorem~\ref{thm:prime}.
Let $[a,b] \subseteq [-1,1]$ be a subinterval. Then using the same arguments
as in the proof of Theorem~\ref{thm:prime} one can show that the set of primes
$\{p \;|\; \frac{a(tp^2)}{2a(t)p^{k-1/2}} \in [a,b]\}$ has a natural
density equal to $\mu([a,b])$.
Similarly, a more general version of Theorem~\ref{thm:prime-regular} also holds.
\end{rem}

\section{Equidistribution of signs for $\{a(tn^2)\}_{n \in \NN}$}\label{sec:all}

In this section we prove our equidistribution result for the signs of the
coefficients $a(tn^2)$, when $n$ runs through the natural numbers. The same arguments
also work for $n$ running through $k$-free positive integers for any $k \ge 1$.

\begin{thm}\label{thm:squarefree2}
We make the same assumptions as in Theorem~\ref{thm:prime}.
As in Theorem~\ref{thm:prime-regular} we additionally assume that there are
$C>0$ and $\alpha > 0$ such that for all subintervals $[a,b] \in [-1,1]$ one has
$$ \left| \frac{ \# \{p\le x \textnormal{ prime }| \; \frac{A_t(p)}{a(t)2p^{k-1/2}} \in [a,b]\}}{\pi (x)}
 - \mu([a,b]) \right| \le \frac{C}{x^\alpha}.$$
We also assume $a(t)>0$.
We define the multiplicative (but, not completely multiplicative) function
$$s(n) =
 \begin{cases}
  1 & \text{if $a(tn^2)>0$,}\\
 -1 & \text{if $a(tn^2)<0$,}\\
  0 & \text{if $a(tn^2)=0$.}
 \end{cases}$$
Let $S(z):=\sum_{n=1}^{\infty} \frac{s(n)}{n^z}$ be the Dirichlet series of~$s(n)$.

Then $S(z)$ is holomorphic for $\Real(z) \ge 1$.
\end{thm}

\begin{proof}
Since $s(n)$ is multiplicative because of Equation \ref{coef} and the assumption $a(t)>0$,
we have the Euler product $S(z)=\prod_{p \in \PP}^{}\sum_{k=0}^{\infty}s(p^k)p^{-kz}$.
Taking logarithm of $S(z)$, we obtain
\begin{multline*}
\log S(z)
=\sum_{p \in \PP_{>0}} \log \big(1+ \frac{1}{p^z}+g(z,p)\big)
+\sum_{p \in \PP_{<0}} \log \big(1- \frac{1}{p^z}+g(z,p)\big)\\
+\sum_{p \in \PP_{=0}} \log \big(1               +g(z,p)\big),
\end{multline*}
where $g(z,p) := \sum_{m=2}^\infty\frac{s(p^m)}{p^{mz}}$ is a holomorphic function on
$\Real(z) \ge 1$ in the variable~$z$ for fixed~$p$.
Note that for all $p$, we have $|g(z,p)| \leq \frac{1}{p^z(p^z -1)}$.
Using this, we conclude that 
$$ \log S(z) = \sum_{p \in \PP_{>0}} \frac{1}{p^z} - \sum_{p \in \PP_{<0}} \frac{1}{p^z}+k(z), $$
where $k(z)$ is holomorphic function for $\Real(z) \ge 1$.

Since $\PP_{>0}$ and $\PP_{<0}$ are regular set of primes having the same density $\frac{1}{2}$ by Theorem~\ref{thm:prime-regular}, we can conclude that $\log S(z)$ is analytic for $\Real(z) \ge 1$ by Lemma~\ref{lem:regular}. Since the exponential function is holomorphic, we find that $S(z)$ is analytic for $\Real(z) \ge 1$ by taking the exponential of $\log S(z)$, as was to be shown.
\end{proof}

We now deduce our main density statement from Theorem~\ref{thm:squarefree2}.

\begin{cor}\label{cor:squarefree2}
Assume the setting of Theorem~\ref{thm:squarefree2}.
Then the sets
$$ \{n \in \NN \;|\; a(tn^2) > 0 \}
\textrm{ and } \{n \in \NN \;|\;  a(tn^2) < 0 \}$$
have equal positive Dedekind-Dirichlet densities, that is,
both are precisely half of the Dedekind-Dirichlet density of the set
$ \{n \in \NN \;|\; a(tn^2) \neq 0 \}$.
\end{cor}

\begin{proof}
We assume without loss of generality $a(t)>0$, since the statement is invariant under
replacing $f$ by~$-f$. We use the notation of Theorem~\ref{thm:squarefree2}.
Since the set of all natural numbers has a Dedekind-Dirichlet density of $1$ (since the
Riemann-zeta function has a pole of order $1$ and residue~$1$ at~$1$), we can write
$$\lim_{z \to 1^+} (z-1)\sum_{n \in \NN }^{}\frac{1}{n^z}=1.$$
Since $S(z)= \sum_{a(tn^2)>0} \frac{1}{n^z} - \sum_{a(tn^2)<0} \frac{1}{n^z}$ is holomorphic for $\Real(z) \ge 1$ by Theorem \ref{thm:squarefree2} (in fact, much less suffices), we deduce
$$\lim_{z \to 1^+} (z-1)\left[2 \sum_{a(tn^2)>0}^{}\frac{1}{n^z}+\sum_{a(tn^2)=0}^{}\frac{1}{n^z}\right]=1.$$
Define $t(n) :=[s(n)]^2$ and $T(z):=\sum_{n=1}^{\infty}\frac{t(n)}{n^z}$.
Since $t(n)$ is multiplicative, $T(z)$ has an Euler product, which we use now.
Put $A(z):=\frac{T(z)}{\zeta(z)}$, then
\begin{align*}
A(z)
&= \prod_{p}^{} (1-\frac{1}{p^z}). \prod_{p}^{}\left(1+\sum_{n=1, a(tp^{2n}) \neq 0}^{\infty} \frac{1}{p^z}\right)\\
&= \prod_{p, a(tp^2) \neq 0}^{}\left((1-\frac{1}{p^z})(1+\frac{1}{p^z}+\sum_{n=2,a(tp^{2n}) \neq 0}^{\infty}\frac{1}{p^{nz}})\right)\\
&\cdot \prod_{p, a(tp^2) = 0}^{}\left((1-\frac{1}{p^z})(1+\sum_{n=2}^{\infty}\frac{1}{p^{nz}})\right) \\
&= \prod_{p,a(tp^2)\neq 0}^{} \left(1- \frac{1}{p^{2z}}+r_1(z,p)\right).\prod_{p,a(tp^2)=0}^{} \left(1-\frac{1}{p^z}+r_2(z,p)\right),
\end{align*}
where $r_1(z,p)$ and $r_2(z,p)$ are the remaining terms. 
Taking the logarithm of $A(z)$, we conclude that $\sum_{p,a(tp^2)=0}^{} \log \left(1-\frac{1}{p^z}+r_2(z,p)\right)$ is holomorphic on $\Real(z) \ge 1$, since $\PP_{=0}$ is a regular set of primes of density $0$, by Theorem \ref{thm:prime-regular}.
Moreover $\sum_{p,a(tp^2)\neq 0}^{} \log\left(1- \frac{1}{p^{2z}}+r_1(z,p)\right)$ is also holomorphic on $\Real(z) \ge 1$.
Taking the exponential shows that $A(z)=\frac{T(z)}{\zeta(z)}$ is holomorphic for $\Real(z) \ge 1$.
We obtain $A(1)>0$ and $\lim_{z\to 1^+} (z-1)T(z)=A(1)$.
Therefore we conclude that the set $\{n \in \NN \;|\; a(tn^2) \neq 0 \}$ has a
Dedekind-Dirichlet density equal to $A(1)$. Hence, the limit
$\lim_{z\to 1^+} (z-1)\sum_{a(tn^2)=0 }^{}\frac{1}{n^z}=1-A(1)$
exists.
So we conclude that 
$$\lim_{z \to 1^+}(z-1)\sum_{a(tn^2)>0 }^{}\frac{1}{n^z}=\frac{A(1)}{2}.$$
This implies that the Dedekind-Dirichlet density of the two sets in the statement are equal and completes the proof.
\end{proof}

\subsection*{Acknowledgements}

I.I.\ was supported by The Scientific and Technological Research Council of Turkey (TUBITAK). I.I would also like to thank the University of Luxembourg for having
hosted him as a postdoctoral fellow of TUBITAK
G.~W.\ acknowledges partial support by the priority program~1489 of the
Deutsche Forschungsgemeinschaft (DFG) and by the Fonds National de la
Recherche Luxembourg (INTER/DFG/12/10).

The authors would like to thank Winfried Kohnen and Jan Bruinier for having suggested
the problem and encouraged them to write this article.

\end{document}